\documentclass[12pt,reqno]{amsart}
\usepackage{hyperref}
\usepackage{amsmath,amsthm,amsopn,amssymb,varioref}
\usepackage{color, bm, amscd, tikz-cd}
\newcommand{\cE}{{\mathcal E}}
\newcommand{\cF}{{\mathcal F}}

\newcommand{\cH}{{\mathcal H}}

\newcommand{\cK}{{\mathcal K}}
\newcommand{\cL}{{\mathcal L}}

\newcommand{\cP}{{\mathcal P}}

\newcommand{\cT}{{\mathcal T}}
\newcommand{\cU}{{\mathcal U}}

\newcommand{\cW}{{\mathcal W}}

\newcommand{\fJ}{{\mathfrak J}}

\newcommand{\fT}{{\mathfrak T}}
\newcommand{\fP}{{\mathfrak P}}
\newcommand{\sbm}[1]{\left[\begin{smallmatrix} #1
		\end{smallmatrix}\right]}

\newtheorem{thm}{Theorem}[section]

\newtheorem{corollary}[thm]{Corollary}
\newtheorem{lemma}[thm]{Lemma}

\newtheorem{proposition}[thm]{Proposition}

\usepackage{geometry}

\newtheorem{definition}[thm]{Definition}
\newtheorem{remark}[thm]{Remark}
\newtheorem{remarks}[thm]{Remarks}

\numberwithin{equation}{section}

\def\textmatrix#1&#2\\#3&#4\\{\bigl({#1 \atop #3}\ {#2 \atop #4}\bigr)}
\def\dispmatrix#1&#2\\#3&#4\\{\left({#1 \atop #3}\ {#2 \atop #4}\right)}
\numberwithin{equation}{section}

\def\textmatrix#1&#2\\#3&#4\\{\bigl({#1 \atop #3}\ {#2 \atop #4}\bigr)}
\def\dispmatrix#1&#2\\#3&#4\\{\left({#1 \atop #3}\ {#2 \atop #4}\right)}

\begin{document}

\title[Toeplitz operators and pseudo-extensions]{ Toeplitz operators and pseudo-extensions}

\author[Bhattacharyya]{Tirthankar Bhattacharyya}
\address[Bhattacharyya]{Department of Mathematics, Indian Institute of Science, Bangalore 560 012, India.}
\email{tirtha@iisc.ac.in}

\author[Das]{B. Krishna Das}
\address[Das]{Department of Mathematics, Indian Institute of Technology Bombay, Powai, Mumbai, 400076, India.}
\email{dasb@math.iitb.ac.in, bata436@gmail.com}

\author[Sau]{Haripada Sau}
\address[Sau]{Department of Mathematics,  Indian Institute of Technology Guwahati, Guwahati, Assam 781039, India.}
\email{sau@vt.edu, haripadasau215@gmail.com}


\begin{abstract}
There are three main results in this paper. First, we find an easily computable and simple condition which is  necessary and sufficient for a commuting tuple of contractions to possess a non-zero Toeplitz operator. This condition is just that the adjoint of the product of the contractions is not pure. On one hand this brings out the importance of the product of the contractions and on the other hand, the non-pureness turns out to be equivalent to the existence of a pseudo-extension to a tuple of commuting unitaries. The second main result is a commutant pseudo-extension theorem obtained by studying the unique canonical unitary pseudo-extension of a tuple of commuting contractions. The third one is about the $C^*$-algebra generated by the Toeplitz operators determined by a commuting tuple of contractions. With the help of a special completely positive map, a different proof of the existence of the unique canonical unitary pseudo-extension is given.

\end{abstract}
\maketitle

\section{Introduction}

A contraction $P$ acting on a Hilbert space is called $pure$ if $P^{*n}$ converges to zero strongly as $n \rightarrow \infty$.

Let $\mathbb D$ be the open unit disk while $\mathbb D^d$, $\overline{\mathbb D}^d$ and $\mathbb T^d$ denote the open polydisk, the closed polydisk, and the $d$-torus, respectively in $d$-dimensional complex plane for $d\ge 2$.

The seminal paper \cite{BH} of Brown and Halmos introduced the study of those operators $X$ on the Hardy space which satisfy $M_z^*XM_z = X$ where $M_z$ is the unilateral shift on the Hardy space. These are called $Toeplitz$ operators and have been greatly studied. Among the many directions in which Toeplitz operators have been generalized, operators $X$ on a Hilbert space $H$ that satisfy $P^*XP = X$ for a contraction $P$ on $H$ hold a prime place. Prunaru generalized this to study Toeplitz operators corresponding to a commuting contractive tuple (also called a $d$-contraction) in \cite{PrunaruJFA}. Prunaru's techniques are specific to the Euclidean unit ball.

In connection with the polydisk, the Toeplitz operators that have been well studied are those which satisfy
$$
M_{z_j}^*XM_{z_j}=X \text{ for each } j=1,2,\dots,d,
$$where $M_{z_j}$ is multiplication by the coordinate function `$z_j$' on $H^2(\mathbb D^d)$, the Hardy space over $\mathbb D^d$. The class of these Toeplitz operators is large and has been studied greatly, see \cite{CKL} and the references therein. Thus the following definition is natural.
\begin{definition}
Let $\underline{T}=(T_1, T_2, \dots,T_d)$ be a commuting tuple of contractions on a Hilbert space $\mathcal H$. A bounded operator $A$ on $\mathcal H$ is said to be a {\em $\underline{T}$-Toeplitz operator} if it satisfies {\em Brown-Halmos relations} with respect to $\underline{T}$, i.e.,
\begin{eqnarray}\label{Gen-Brown-Halmos-Poly}
T_i^*AT_i=A \text{ for each $1\leq i \leq d$}.
\end{eqnarray}
The $*$--closed and norm closed vector space of all $\underline{T}$-Toeplitz operators is denoted by $\mathcal T(\underline{T})$.
\end{definition}
One of the aims of this note is to answer when this vector space $\mathcal T(\underline{T})$ is {\em non-trivial}, i.e., contains a non-zero operator. The prime tool for deciding this question is the product operator.

For a $d$-tuple $\underline{T}=(T_1,T_2,\dots,T_d)$ of commuting contractions on a Hilbert space $\cH$, the contraction $P=T_1T_2\cdots T_d$ will be refereed to as the {\em product contraction} of $\underline{T}$.


A remarkable fact in the theory of Hilbert space operators says that a commuting tuple of isometries extends to a commuting tuple of unitaries. This is true, in particular, for the shifts $M_{z_j}$ on the Hardy space of the polydisk. A natural question then arises. Is there a connection between the richness of the class of Toeplitz operators $\mathcal T(M_{z_1}, M_{z_2}, \ldots , M_{z_d})$ on the Hardy space of the polydisk and the fact that the tuple $(M_{z_1}, M_{z_2}, \ldots , M_{z_d})$ extends to commuting unitaries? This motivates the definition below and the theorem following it.

\begin{definition}\label{D:Ext-PDisk}
Let $\underline{T}=(T_1,T_2,\dots,T_d)$ be a $d$-tuple of commuting bounded operators on a Hilbert space $\cH$. A $d$-tuple $\underline{U}=(U_1,U_2,\dots,U_d)$ of commuting bounded operators on a Hilbert space $\cK$ is called a pseudo-extension of $\underline{T}$, if
\begin{itemize}
\item[(1)] there is a non-zero contraction $\fJ: \cH\to \cK$, and
\item[(2)] $\fJ T_j=U_j\fJ$, for every $j=1,2,\dots,d$.
\end{itemize}
We denote such a pseudo-extension of $\underline T$ by $(\fJ,\cK, \underline{U})$.

A pseudo-extension $(\fJ,\cK,\underline{U})$ of $T$
 is said to be minimal if $\cK$ is the smallest reducing space for each $U_j$ containing $\fJ\cH$. We say that two pseudo-extensions $(\fJ,\cK,\underline{U})$ and $(\tilde\fJ,\tilde\cK,\tilde{\underline{U}})$ of $\underline{T}$ are unitarily equivalent if there exists a unitary $W:\cK\to\tilde\cK$ such that
\begin{align*}
W U_j = \tilde{U}_j W \text{ for all } j=1,2 \ldots , d \text{ and } W\fJ=\tilde\fJ.
\end{align*}

A minimal pseudo-extension $(\fJ,\cK,\underline{U})$ of $\underline T$ is called canonical if
\begin{align}\label{ContEmbedd}
\fJ^*\fJ=\operatorname{SOT-}\lim P^{*n}P^n.
\end{align}

\end{definition}

The role of the contraction $\fJ$ may need to be emphasized at times and then we shall say that $\underline{U}$ is a pseudo-extension of $\underline{T}$ through $\fJ$. The condition (2) in the above definition implies that each $U_j$ is an extension of $\tilde T_j:\overline{\operatorname{Ran}}\;\fJ\cH\to\overline{\operatorname{Ran}}\;\fJ\cH$ densely defined as
\begin{align*}
\tilde T_j (\fJ h):=\fJ T_jh, \text{ for every }h\in\cH.
\end{align*}
This is why we call $\underline U$ is a pseudo-extension of $\underline T$.

A tuple of commuting contractions on a Hilbert space does not possess a unitary extension, in general. However, existence of a unitary pseudo-extension for a $d$-tuple of commuting contraction $\underline T$
can now be characterized in terms of a condition on the product contraction $P$ of $\underline T$. This is also intimately related to the non-triviality of
 $\mathcal T(\underline T)$.


\begin{thm}\label{Thm:Ext} Let $\underline{T}=(T_1,T_2,\dots,T_{d})$ be a $d$-tuple of commuting contractions on a Hilbert space $\mathcal H$. Then the following are equivalent.

 \begin{enumerate}
 \item $\mathcal T(\underline{T})$ is non-trivial.
 \item The adjoint of the product contraction $P$ of $\underline{T}$ is not pure, i.e., $P^n\nrightarrow 0$ strongly.
 \item There exists a unique (up to unitary equivalence) canonical unitary pseudo-extension of the tuple $\underline{T}$.

\end{enumerate}

\end{thm}

This theorem is proved in section 2.

A fundamental concept, called dilation, introduced by Sz.-Nagy has stimulated extensive research in operator theory.

\begin{definition}\label{dilation}
Let $\underline{T}=(T_1,T_2,\dots,T_d)$ be a $d$-tuple of commuting bounded operators on a Hilbert space $\cH$. A $d$-tuple $\underline{V}=(V_1,V_2,\dots,V_d)$ of commuting bounded operators on a Hilbert space $\cK$ is called a dilation of $\underline{T}$, if $\cH$ is a subspace of $\cK$ and $V_i^*|_\cH = T_i^*$ for $i=1,2, \ldots ,d$. The dilation is called isometric if $V_i$ are isometries. \end{definition}

It is well-known that a commuting tuple of contractions does not have a commuting isometric dilation in general. In case $\underline{T}^*$ has a commuting isometric dilation can we talk of the unitary part of the isometric dilation tuple and is that then an example of a pseudo-extension to a tuple of commuting unitaries? This question has a gratifying answer. Recall that the classical Wold decomposition \cite{Wold} states that any isometry $V$ acting on a Hilbert space $\cH$ is unitarily equivalent to the direct sum of a unilateral shift $M_z$ of multiplicity equal to $\dim(I_{\cH}-VV^*)$ and a unitary operator $U$. The unitary operator $U$ is often regarded as the `unitary part' of the isometry $V$. Several attempts have been made to obtain a multivariable analogue of Wold decomposition, see \cite{Burdak, SarkarLAA, Slo1980, Slo1985} and references therein. Perhaps the most elegant among these models is the one obtained by Berger, Coburn and Lebow \cite{BCL}, see Theorem \ref{Thm:BCL}. We shall use its elegance to analogously define the {\em unitary part} of a commuting tuple of isometries. Then we shall answer the question above affirmatively in Theorem~\ref{non-canonical}.

The relation between the existence of a non-zero operator in $\mathcal T(\underline{T})$ and the existence of a pseudo-extension $(\fJ,\cK,\underline{U})$ of $\underline T$ goes much further. A study of the unital $C^*$-algebra $\mathcal C$ generated by $I_{\cH}$ and $\cT(\underline{T})$ reveals that it has a $*$-representation $\pi$ onto the commutant of $\underline{U}$, denoted by $\underline{U}'$. In fact, there exists a natural {\em completely isometric cross section} $\rho$ of the $*$-representation $\pi$ that maps onto $\cT(\underline{T})$. This in turn proves that $\cT(\underline{T})$ and $\underline{U}'$ are in one-to-one correspondence. Furthermore, we prove that every element $X$ in $\underline{T}'$, the commutant of $\underline{T}$, can be $\fJ$-extended to an element $\Theta(X)$ in $\underline{U}'$ and that the correspondence
$$ X\mapsto\Theta(X) $$
is completely contractive, unital and multiplicative. This is the content of Theorem \ref{Thm:ComLftPDisk}.

\section{$\underline{T}$-Toeplitz operators and pseudo-extensions}\label{S:PseudoExt}
This section has the  proof of Theorem \ref{Thm:Ext}. We shall take up the path $(1)\Longrightarrow(2)\Longrightarrow(3)\Longrightarrow(1)$.

\begin{proof}[{\bf Proof of $(1)\Rightarrow(2)$:}]
Let there be a non-zero $\underline{T}$-Toeplitz operator $A$. This means that $T_j^*AT_j = A$ for all $j=1,2, \ldots , d$. This implies $P^*AP = A$ where $P$ is the product contraction. Thus, for all $n \geq 0$ we have $A = P^{*n}AP^n$ and hence $\| Ah \| \le \| A\| \| P^nh\| $ for every vector $h$. So, if $P^n$ strongly converges to $0$, then $A = 0$ which is a contradiction.

For two hermitian operators $T_1$ and $T_2$, we say that $T_1 \preceq T_2$ if $T_2 - T_1$ is a positive operator.
The following well known result called Douglas's Lemma has found many applications.
\begin{lemma}\label{L:DougLem}[Theorem 1, \cite{Douglas}]
Let $A$ and $B$ be two bounded operators on a Hilbert space $\mathcal{H}$. Then there exists a contraction $C$ such that
$A=BC$ if and only if $$AA^*\preceq BB^*.$$
\end{lemma}
The proof is easy. Indeed, defining $C^*$ on the range of $B^*$ as $C^*B^*x = A^*x$ i all that is required. We shall need it below.

\noindent{\bf Proof of $(2)\Rightarrow(3)$:} Let $\underline{T}=(T_1,T_2,\dots,T_d)$ be a $d$-tuple of commuting contractions such that $P^n\nrightarrow0$ strongly. As $P$ is a contraction
$$
 I_{\mathcal{H}}\succeq P^*P\succeq  P^{*2}P^2\succeq \cdots\succeq {P^*}^nP^n\succeq \cdots\succeq 0.
$$ This guarantees a positive contraction $Q$ such that
\begin{align}\label{assymplimit}
Q=\operatorname{SOT-}\lim P^{*n}P^n.
\end{align}
The hypothesis makes $Q$ non-zero. From the above expression of $Q$ one can read off the validity of
$$P^*QP= Q.$$ Hence we can define an isometry $X : \overline{\operatorname{Ran}}Q \rightarrow \overline{\operatorname{Ran}}Q $ satisfying
\begin{align}\label{X}
X:Q^\frac{1}{2}h \mapsto Q^\frac{1}{2}Ph \text{ for each }h\in\cH.
\end{align}
We note that \begin{align}\label{ineq1}
T_j^*QT_j \preceq Q \text{ for each }j=1,2,\dots,d.
\end{align}
Indeed, since $P$ is the product contraction, we get for each $j=1,2,\dots,d$,
\begin{align*}
\langle T_j^*QT_jh,h\rangle = \lim_n\langle P^{*n}(T_j^*T_j)P^nh,h \rangle
 \leq \lim_n\langle P^{*n}P^{n}h,h\rangle = \langle Qh,h\rangle.
\end{align*}
By Douglas's Lemma \ref{L:DougLem}, we obtain a contraction $X_j : \overline{\operatorname{Ran}}Q \rightarrow \overline{\operatorname{Ran}}Q$ such that for every $h\in\cH$,
\begin{align}\label{Xj}
X_j:Q^\frac{1}{2}h \mapsto Q^\frac{1}{2}T_jh \text{ for each $j=1,2,\dots,d$}.
\end{align}
The contractions $X_j$ are commuting because using the commutativity of $\underline{T}$ we get for each $i,j=1,2,\dots,d$ and $h\in\cH$,
\begin{align*}
X_iX_j Q^\frac{1}{2}h = X_iQ^\frac{1}{2}T_jh &= Q^\frac{1}{2}T_iT_jh\\
&=Q^\frac{1}{2}T_jT_ih=X_jQ^\frac{1}{2}T_ih=X_jX_iQ^\frac{1}{2}h.
\end{align*}Since $P$ is the product contraction, a computation similar to the one above yields $X=X_1X_2\cdots X_d$. But $X$ is an isometry. So all of its commuting factors have to be isometries and hence the contractions $X_j$ have to be isometries. Let $\underline{W}=(W_1,W_2,\dots,W_d)$ acting on $\cK$ be a minimal unitary extension of $\underline{X}$. Define a contraction $\fJ:\cH\to\cK$ as
\begin{align*}
\fJ :h\mapsto Q^\frac{1}{2}h \text{ for every } h \in \cH.
\end{align*}
The computation below shows that $\fJ$ intertwines each $W_j$ with $T_j$:
 \begin{align}\label{Int}
W_j\fJ h=W_jQ^\frac{1}{2}h=X_jQ^\frac{1}{2}h=Q^\frac{1}{2}T_jh=\fJ T_jh.
\end{align}Finally, by definition of $\fJ$ and $Q$, it follows that $\fJ^*\fJ$ is the limit of $P^{*n}P^n$ in the strong operator topology and hence $(\fJ, \cK,\underline{W})$ is a canonical pseudo-extension of $\underline{T}$.

For the uniqueness part, let us suppose that $(\fJ,\cK,{\underline{U}}=( U_1,\dots, U_d))$ and $(\tilde\fJ,\tilde\cK,\tilde{\underline{U}}=(\tilde U_1,\dots,\tilde U_d))$ be two canonical unitary pseudo-extensions of $\underline{T}$. We show that these two are unitarily equivalent. To that end, let us define the operator $\tau:\mathcal K \to \tilde{\mathcal K}$ densely by
$$
\tau:f(\underline{U}, \underline{U}^*)\fJ h\mapsto f(\underline{\tilde{U}}, \underline{\tilde{U}}^*) \tilde\fJ  h
$$
for every $h\in \mathcal H$ and polynomial $f$ in $\bm z$ and $\overline{\bm z}$. Since $(\tilde\fJ,\tilde\cK,\tilde{\underline{U}})$ is minimal, $\tau$ is surjective. Note that $\tau$ clearly satisfies $\tau\fJ=\tilde\fJ$. We will be done if we can show that $\tau$ is an isometry. Let $f$ be a polynomial in $\bm z$ and $\overline{\bm z}$ and $\bar{f}f=\sum a_{\bm n,\bm m}\bm z^{\bm n}\overline{\bm z}^{\bm m}$. Then for every $h\in \cH$,
\begin{align}\label{UniqComp}
\notag \|f(\underline{U}, \underline{U}^*)\fJ h\|^2
 \notag & = \sum a_{\bm n,\bm m}\langle \fJ^*\underline{U}^{* \bm m}\underline{U}^{\bm n}\fJ h, h\rangle\\
\notag & = \sum a_{\bm n,\bm m}\langle \underline{T}^{* \bm m}\fJ^*\fJ\underline{T}^{\bm n}h, h\rangle\\
 & = \sum a_{\bm n,\bm m}\langle \underline{T}^{* \bm m} Q\underline{T}^{\bm n}h, h\rangle.
\end{align}
Since the last term only depends on the $d$-tuple $\underline{T}$, $\tau$ is an isometry.

\noindent{\bf Proof of $(3)\Rightarrow(1)$:} Note that if a $d$-tuple $\underline{T} = (T_1, T_2, \ldots , T_d)$ of commuting contractions has even an isometric pseudo-extension $\underline{V} = (V_1, V_2, \ldots , V_d)$ through $\fJ$, then for all $j=1,\cdots, d$
$$ T_j^* \fJ^*\fJ T_j = \fJ^* V_j^* V_j \fJ = \fJ^* \fJ.$$
This proves that the non-zero operator $\fJ^* \fJ$ belongs to $\mathcal T(\underline{T})$. This in particular establishes that (3) implies (1).
\end{proof}

\begin{remarks}\label{R:Can-nonCan}
Several remarks are in order.
\begin{itemize}
\item[(1)] It follows from the proof of $(3)\Rightarrow(1)$ of Theorem \ref{Thm:Ext} that if a $d$-tuple $\underline{T}$ of commuting contractions has an isometric pseudo-extension, then it has a canonical unitary pseudo-extension. Indeed, if $(\fP,\cL,\underline{W})$ is any isometric pseudo-extension of $\underline{T}$, then as observed in the proof of $(3)\Rightarrow(1)$ of Theorem \ref{Thm:Ext}, the non-zero operator $\fP^* \fP$ is a $\underline{T}$-Toeplitz operator. Hence by Theorem \ref{Thm:Ext} there exists a canonical unitary pseudo-extension of $\underline{T}$.
\item[(2)] Let $T$ be a contraction acting on a Hilbert space $\cH$. It is known that the minimal unitary (or isometric) dilation space of $T$ is always infinite dimensional even in the case when $\cH$ is finite dimensional. We observe that, unlike the case of the dilation theory, if $\underline{T}$ is a $d$-tuple of commuting contraction acting on a finite dimensional Hilbert space, then the canonical unitary pseudo-extension space for $\underline{T}$ is also finite dimensional. Since any two canonical unitary pseudo-extensions of a given tuple are unitarily equivalent, we consider the canonical unitary pseudo-extension constructed in the proof of $(2)\Rightarrow(3)$ of Theorem \ref{Thm:Ext}. Recall that for each $j=1,2,\dots,d$, the isometry $X_j$ as defined in \eqref{Xj} is itself a unitary because it acts on a finite dimensional space, viz., $\overline{\operatorname{Ran}}Q$. Therefore the tuple $X=(X_1,X_2,\dots,X_d)$ acting on $\overline{\operatorname{Ran}}Q$ is a canonical unitary pseudo-extension of $\underline{T}$.
\item[(3)] We also observe that a $d$-tuple $\underline{T}$ of commuting contractions has a unitary pseudo-extension through an isometry $\fJ $ if and only if $\underline{T}$ is a commuting tuple of isometries. Thus, Theorem \ref{Thm:Ext} subsumes the standard extension of commuting isometries to commuting unitaries as a special case.
\end{itemize}
\end{remarks}

We now link pseudo-extension of $\underline{T}$ with isometric dilation of $\underline{T}^*$ when it exists. To that end, we need an old result of Berger, Coburn and Lebow which has gained a lot of attention recently. Indeed it is the result of Berger, Coburn and Lebow that inspired explicit constructions of And\^o dilation in \cite{DSS-Adv2018} for a special case and then in \cite{BS-Ando} for the general case.

\begin{thm}[Theorem 3.1, \cite{BCL}] \label{Thm:BCL}
Let $(V_1,V_2,\dots,V_d)$ be a $d$-tuple of commuting isometries acting on a Hilbert space $\cK$. Then there exist Hilbert spaces $\cE$ and $\cF$, unitary operators $\cU=\{U_1, \dots, U_d\}$ and projection operators $\cP=\{P_1, \dots, P_d\}$ acting on $\cE$, and commuting unitary operators $\cW=\{W_1, \dots, W_d\}$ acting on $\cF$ such that $\cK$ can be decomposed as
   \begin{align}\label{VWold}
   \cK=H^2(\cE)\oplus\cF
   \end{align}
and with respect to this decomposition
\begin{align}  \label{BCL1}
&V_j = M_{U_j(P_j^\perp+zP_j)}\oplus W_j,\;
V_{(j)}=M_{(P_j+zP_j^\perp) U_j^*}\oplus W_{(j)}  \text{ for }1 \le j \le d,\\
&\text{ and }  V =V_1V_2\cdots V_d= M_z\oplus W_1W_2\cdots W_d,
\label{WoldV}
\end{align}
where $V_{(j)}=\prod\limits_{i\neq j} V_i$ and $W_{(j)}=\prod\limits_{i\neq j}W_i$.
\end{thm}
The decomposition  \eqref{WoldV} of the product isometry $V=V_1V_2\cdots V_d$ with respect to ~\eqref{VWold} is actually the same as the Wold decomposition of $V$.  It is remarkable that the Wold decomposition of $V$ reduces each of its commuting factors into the direct sum of two operators.
\begin{definition}
For a $d$-tuple $\underline{V}=(V_1,V_2,\dots,V_d)$ of commuting isometries, the $d$-tuple $\cW=(W_1,W_2,\dots,W_d)$ of commuting unitaries obtained in Theorem \ref{Thm:BCL} is called the unitary part of $\underline{V}$.
\end{definition}
The following theorem relates pseudo-extensions with dilation theory and also
provide examples of non-canonical pseudo-extensions.
\begin{thm}\label{non-canonical}
For a $d$-tuple of commuting contractions $\underline T$ on $\mathcal H$, if $\underline{T}^*$ has a minimal isometric dilation $\underline{V}=(V_1,V_2,\dots,V_d)$ on $\cK$ with non-zero unitary part $\underline{U}=(U_1,U_2,\dots,U_d)$ acting on $\cF\subseteq \cK$ then $\underline{U}$ is a
unitary pseudo-extension of $\underline T$.
\end{thm}
\begin{proof}
To prove $\underline U$ is a unitary pseudo-extension of $\underline T$, the required contraction $ \fJ:\mathcal H \to \cK $ is defined as
$$
\fJ:h\to P_{\cF}h,\quad (h\in\mathcal H)
$$where $P_\cF$ denotes the orthogonal projection of $\cK$ onto $\cF$.  Since $\underline{V}$ is minimal, $\cF$ cannot be orthogonal to $\cH$ and hence $\fJ$ is non-zero. Since each $V_j^*$ is an extension of $T_j$ and since $\cF$ is reducing for each $V_j$, we get
\begin{align*}
U_j^*\fJ h=V_j^*P_{\cF}h=P_{\cF}V_j^*h=P_{\cF}T_jh=\fJ T_jh
\end{align*}
for each $h$ in $\cH$. This completes the proof.
\end{proof}
\begin{remark}
We observed that the unitary pseudo-extension obtained in Theorem~\ref{non-canonical} is non-canonical, in general,
because the contraction $\fJ$ need not satisfy \eqref{ContEmbedd}.
 We remark here that for $d=2$, there is an explicit construction of dilation whose unitary part gives rise to the canonical unitary pseudo-extension, see Theorem 3 of \cite{BS-Ando}.

\end{remark}


  From the above theorem follows the following corollary.
\begin{corollary}\label{P:Aux}
Let $\underline{T}$ be a $d$-tuple of commuting contractions such that
 \begin{enumerate}
 \item $P^n\to0$ strongly and
  \item $\underline{T}^*$ has an isometric dilation. \end{enumerate}
 Then the unitary part of the minimal isometric dilation of $\underline{T}^*$ is zero.
\end{corollary}
\begin{proof}
Let $\underline{V}$ be a minimal isometric dilation of $\underline{T}^*$. If the unitary part $\underline{U}$ of $\underline{V}$ is non-zero, then by the above discussion $\underline{U}^*$ is a pseudo-extension of $\underline{T}$. This contradicts the fact that $P^n\nrightarrow0$ is a necessary and sufficient condition for existence of a pseudo-extension $\underline{T}$.
\end{proof}

%
We end this section by establishing a relation between a non-canonical unitary pseudo-extension and the canonical unitary pseudo-extension of a given tuple of commuting contractions. It shows that any unitary pseudo-extension of a given tuple of commuting contractions factors through the canonical unitary pseudo-extension.
\begin{proposition}
Let $\underline{T}$ be a $d$-tuple of commuting contractions acting on a Hilbert space $\cH$ such that $P^n\nrightarrow 0$ strongly as $n\to \infty$. Let $(\fP,\cL,\underline{W})$ be a
unitary pseudo-extension of $\underline{T}$. If $(\fJ,\cK,\underline{U})$ is the canonical pseudo-extension of $\underline{T}$, then
\begin{enumerate}
\item $\fP^*\fP\leq\operatorname{SOT-}\lim P^{*n}P^n=\fJ^*\fJ$ and

\item $\underline{W}$ is a unitary pseudo-extension of $\underline U$
through a contraction $\fT:\cK\to\cL$ such that $\fT\fJ=\fP$.
\end{enumerate}
\end{proposition}
\begin{proof}
We have seen in the proof of $(3)\Rightarrow(1)$ of Theorem \ref{Thm:Ext} that if $(\fP,\cL,\underline{W})$ is a unitary pseudo-extension of $\underline{T}$, then $\fP^*\fP$ is a $\underline{T}$-Toeplitz operator. In particular, $\fP^* \fP$ is in $\cT(P)$. This implies
 \begin{align*}
 \fP^*\fP=P^{*n}\fP^*\fP P^n \leq P^{* n}P^n \text{ for every } n.
 \end{align*}
This proves part (1) of the proposition.

For part (2) we define the operator $\fT:\mathcal K \to \cL$ densely by
$$
\fT:f(\underline{U}, \underline{U}^*)\fJ h\mapsto f(\underline{W}, \underline{W^*}) \fP  h
$$
for every $h\in \mathcal H$ and polynomial $f$ in $\bm z$ and $\overline{\bm z}$. Using part (1) of the proposition, a similar computation as done in \eqref{UniqComp} yields
\begin{align*}
 \|f(\underline{W}, \underline{W}^*)\fP h\|\leq  \|f(\underline{U}, \underline{U}^*)\fJ h\| \text{ for every }h\in\cH.
\end{align*}
This shows that $\fT$ is not only well-defined but also a contraction. Finally, it readily follows from the definition of $\fT$ that it intertwines $\underline{U}$ and $\underline{W}$ and that $\fT\fJ=\fP$.
\end{proof}

\section{A commutant pseudo-extension theorem}
The classical commutant lifting theorem -- first by Sarason \cite{Sarason} for a special case and later by Sz.-Nagy--Foias (see Theorem 2.3 in \cite{Nagy-Foias}) for the general case -- is a profound operator theoretic result with wide-ranging applications especially in the theory of interpolation. The most general form of this result states that {\em if $T$ is a contraction with $V$ as its minimal isometric dilation, then any bounded operator $X$ commuting with $T$ has a norm-preserving lifting to an operator $Y$ that commutes with $V$.} Here a {\em lifting} is defined to be a co-extension. In this section, we prove a version of the commutant lifting theorem, herein called commutant pseudo-extension theorem.
\begin{thm}\label{Thm:PLT}
Let $\underline{T}$ be a commuting tuple of contractions and $(\fJ,\cK,\underline{U})$ be its canonical unitary pseudo-extension. Then every $X$ in the commutant of $\underline{T}$ has a pseudo-extension to $Y$ in the commutant of $\underline{U}$ such that $\|Y\|\leq \|X\|$.
\end{thm}
\begin{proof}
Let $P$ be the product contraction of $\underline{T}$ and $Q$ be the limit as in \eqref{assymplimit}. The idea is to obtain a bounded operator $\tilde X$ acting on $\overline{\operatorname{Ran}}Q$ commuting with each isometry $X_j$ as defined in \eqref{Xj} with norm no greater than $\|X\|$ and then apply the standard commutant extension theorem for commuting isometries.

We first do a simple inner product computation. For every $h\in\cH$
\begin{align*}
\|Q^\frac{1}{2}Xh\|^2=\langle X^*QXh,h\rangle=\lim_n\langle P^{*n}X^*XP^nh,h \rangle\leq\|X\|^2\langle Qh,h\rangle.
\end{align*}
Thus there is a bounded operator $\tilde X:\overline{\operatorname{Ran}}Q\to\overline{\operatorname{Ran}}Q$ such that
$$
\tilde X:Q^\frac{1}{2}h\mapsto Q^\frac{1}{2}Xh.
$$with norm at most $\|X\|$. Let $j=1,2,\dots,d$ and $X_j$ be the isometry as defined in \eqref{Xj}, then for each $h\in\cH$,
\begin{align*}
\tilde X X_jQ^\frac{1}{2}h=\tilde X Q^\frac{1}{2}T_jh=Q^\frac{1}{2}XT_jh=Q^\frac{1}{2}T_jXh=X_jQ^\frac{1}{2}Xh=X_j\tilde XQ^\frac{1}{2}h
\end{align*}
showing that $\tilde X$ commutes with the tuple $\underline{X}=(X_1,X_2,\dots,X_d)$ of commuting isometries. We observed in \eqref{Int} that the minimal unitary extension $\underline{W}$ acting on $\cK$ of $\underline{X}$ is actually a canonical unitary pseudo-extension of $\underline{T}$ through a contraction $\fJ:\cH\to\cK$ defined as $\fJ h=Q^\frac{1}{2}h$. Now by a well-known commutant lifting theorem (see, ~\cite[Proposition 10]{Atha}), there exists an operator $Y$ in the commutant of $\underline{W}$ such that $Y|_{\overline{\operatorname{Ran}}Q}=\tilde X$ and $\|Y\|=\|\tilde X\|\leq \|X\|$. Finally to show that $(\fJ,\cK,Y)$ is a pseudo-extension of $X$, we see that for every $h\in\cH$,
\begin{align*}
\fJ Xh=Q^\frac{1}{2}Xh=\tilde XQ^\frac{1}{2}h=YQ^\frac{1}{2}h=Y\fJ h.
\end{align*}This completes the proof.
\end{proof}
The following intertwining pseudo-extension theorem is easily obtained as a corollary to Theorem \ref{Thm:PLT}.
\begin{corollary}
Let $\underline{T}$ and $\underline{T'}$ be two commuting tuples of contractions acting on $\cH$ and $\cH'$, respectively. Let $(\fJ,\cK,\underline{U})$ and $(\fJ',\cK',\underline{U'})$ be their respective canonical unitary pseudo-extensions. Then corresponding to any operator $X:\cH\to\cH'$ intertwining $\underline{T}$ and $\underline{T'}$ there exists another operator $Y:\cK\to\cK'$ such that $Y$ intertwines $\underline{U}$ and $\underline{U'}$, $Y\fJ=\fJ' X$ and $\|Y\|\leq \|X\|$.
\end{corollary}
\begin{proof}
Set
$\tilde X:=\sbm{0&0\\X&0}:\cH\oplus\cH'\to\cH\oplus\cH'$. Then it is easy to see that $\tilde X$ commutes with
$\tilde T_j:=\sbm{T_j&0\\0&T_j'}:\cH\oplus\cH'\to\cH\oplus\cH'$ for each $j=1,2,\dots,d$. Set the unitary operators
$\tilde U_j:=\sbm{U_j&0\\0&U_j'}:\cK\oplus\cK'\to\cK\oplus\cK'\text{ for each }j=1,2,\dots,d$ and denote $\underline{\tilde U}:=(\tilde U_1,\tilde U_2,\dots,\tilde U_d)$. Then by hypothesis it is easy to check that $(\tilde \fJ,\tilde\cK, \underline{\tilde U})$ is a canonical unitary pseudo extension of $\underline{\tilde T}=(\tilde T_1,\tilde T_2,\dots,\tilde T_d)$, where the contraction $\tilde\fJ$ is given by
$$
\tilde\fJ=\sbm{\fJ&0\\0&\fJ'}:\cH\oplus\cH'\to\cK\oplus\cK'=\tilde\cK.
$$By Theorem \ref{Thm:PLT} there exists
$$
\tilde Y=\sbm{Y_{11}&Y_{12}\\Y&Y_{22}}:\cK\oplus\cK'\to\cK\oplus\cK'
$$such that $\tilde Y\underline{\tilde U}=\underline{\tilde U}\tilde Y$, $\tilde\fJ\tilde X=\tilde Y\tilde \fJ$ and $\|\tilde Y\|\leq\|\tilde X\|$. From these relations of $\tilde Y$, it follows that $Y$ has all the desired properties.
\end{proof}
\begin{remark}
One disadvantage in the commutant pseudo-extension theorem is that unlike the classical commutant lifting theorem, the pseudo-extension of a commutant is not norm-preserving, in general and instead the correspondence $X\mapsto Y$ from a commutant to its pseudo-extension is only contractive. We shall see in the next section that this correspondence is actually completely contractive.
\end{remark}

\section{Algebraic structure of the Toeplitz $C^*$-algebra}
For a $d$-tuple $\underline{T}$ of commuting contractions, the {\em Toeplitz $C^*$-algebra}, denote by $C^*(I_{\cH},\cT(\underline{T}))$, is the $C^*$-algebra generated by $I_{\cH}$ and the vector space $\cT(\underline{T})$ of $\underline{T}$-Toeplitz operators. The objective of this section is to study the Toeplitz $C^*$-algebra, which leads to an existential proof of the canonical pseudo-extension of $\underline{T}$.

We begin with a preparatory lemma that gives us a completely positive map with certain special properties that we need. The central idea of the proof goes back to Arveson, see Proposition 5.2 in \cite{Arveson-Nest}. For a subnormal operator tuple, in the multivariable situation, Eschmeier and Everard have proven a similar result by direct construction, see Section 3 of \cite{EE}.

\begin{lemma} \label{L:PJFA}
Let $P$ be a contraction on the Hilbert space $\mathcal H$. Then there exists a completely positive, completely contractive, idempotent linear map
$\Phi : \mathcal B(\mathcal H) \to \mathcal B(\mathcal H)$ such that $Ran \Phi = \mathcal T(P)$.
Moreover, if $ A,B \in \mathcal B(\mathcal H)$ satisfy $P^*(AXB) P = A P^*XPB$ for all $X \in \mathcal B(\mathcal H)$ then $\Phi(AXB) = A\Phi(X)B$. In addition,
$$\Phi(I_\mathcal H ) = Q = \lim_{n\to \infty} P^{*n}P^n$$
where the limit is in the strong operator topology.
\end{lemma}

\begin{proof}
We start by recalling that a Banach limit is a positive linear functional $\mu : l^\infty(\mathbb N) \rightarrow \mathbb C$ which is shift invariant in the sense that
$$ \mu(x_1, x_2, \ldots ) = \mu (x_2, x_3, \ldots )$$
and which extends the natural positive linear functional $x\mapsto \lim_{n\rightarrow \infty} x_n$ defined on the space of convergent sequences. 
For $X$ in $\mathcal B (\mathcal H)$ and vectors $\xi, \eta$ in $\mathcal H$, consider the bounded sesqui-linear form
$$ [\xi, \eta ] = \mu ( \{ \langle P^*XP \xi, \eta \rangle , \langle P^{*2}XP^2 \xi, \eta \rangle , \ldots \} ).$$
Since this form gives rise to a bounded operator, let us call that $\Phi(X)$.
Then $\Phi: X\mapsto \Phi(X)$ defines a linear map on $\mathcal B(\mathcal H)$. Shift invariance of $\mu$ gives us that $\mbox{Ran }\Phi = \mathcal T(P)$. As a consequence, $\Phi$ is idempotent. Other properties of $\Phi$ are straightforward. \end{proof}

The map $\Phi$ obtained above enjoys certain convenient properties as the following lemma shows. We do not prove it because it is part of the proof of Theorem 3.1 in Choi and Effros \cite{CE}. We have singled out what we need.

\begin{lemma}[Choi and Effros] \label{L:CE}
Let $\Phi :\mathcal B(\mathcal H) \to \mathcal B(\mathcal H)$ be a completely positive and completely contractive map
such that $\Phi\circ\Phi=\Phi$. Then for all $X$ and  $Y$ in $\mathcal B(\mathcal H)$ we have
\begin{align}\label{Identities}
\Phi(\Phi(X)Y)= \Phi ( X\Phi(Y)) = \Phi( \Phi(X) \Phi(Y)).
\end{align}
\end{lemma}


We are now ready for the main theorem of this section. The classical Toeplitz operators -- the Toeplitz operators with respect to the unilateral shift on the Hardy space over the unit disk -- are precisely the compressions of the commutant of the minimal unitary extension of the unilateral shift. Part (1) of the following theorem -- the main result of this section -- is a generalization of this result to our context.
\begin{thm}\label{Thm:ComLftPDisk}
Let $\underline{T}=(T_1,T_2,\dots,T_d)$ be a tuple of commuting contractions acting on a Hilbert space $\cH$ such that $P^n\nrightarrow0$. There exists a canonical unitary pseudo-extension $(\fJ,\cK,\underline{U})$ of $\underline{T}$ such that
\begin{enumerate}
\item {\bf Pseudo-compression:} The map $\Gamma$ defined on $\{U_1,\dots,U_{d}\}'$  by $$\Gamma(Y)=\fJ ^*Y\fJ ,$$ is a complete isometry onto $\mathcal T(\underline{T})$;
\item {\bf Representation:} There exists a surjective unital $*$-representation $$\pi:\mathcal C^*\{I_{\mathcal H}, \mathcal T(\underline{T})\}\to \{U_1,\dots,U_{d}\}'$$ such that $\pi \circ \Gamma =I;$
\item {\bf Commutant pseudo-extension:} There exists a completely contractive, unital and multiplicative mapping $$\Theta:\{T_1,\dots,T_{d}\}'\to \{U_1,\dots,U_{d}\}'$$
defined by $\Theta(X)=\pi(\fJ ^*\fJ X)$ which satisfies $$\Theta(X)\fJ =\fJ X.$$
\end{enumerate}
\end{thm}
\begin{proof}
We start with the contraction $P = T_1T_2 \ldots T_d$ and the idempotent, completely positive and completely contractive map $\Phi:\mathcal B(\mathcal H)\to \mathcal B(\mathcal H)$ such that
\begin{align}\label{P-Toep}
\text{Ran}\Phi=\{X\in \mathcal B(\mathcal H):P^*XP=X\}=\mathcal T(P),
\end{align}
as obtained in Lemma~\ref{L:PJFA}.
%
%
Let $\mathcal C^*(I_\mathcal H, \mathcal T(P))$ denote the $C^*$-algebra generated by $\mathcal T(P)$ and $I_{\cH}$. We restrict $\Phi$ to $\mathcal C^*(I_\mathcal H, \mathcal T(P))$ and continue to call it $\Phi$ remembering that the underlying $C^*$-algebra on which it acts is now $\mathcal C^*(I_\mathcal H, \mathcal T(P))$.

  Let $(\mathcal K, \pi, \fJ)$ be the minimal Stinespring dilation of $\Phi$. Thus, $\mathcal K$ is a Hilbert space, $\fJ: \mathcal H \rightarrow \mathcal K$   is a bounded operator and $\pi$ is a unital $*$-representation of $\mathcal C^*(I_\mathcal H, \mathcal T(P))$ taking values in $\mathcal B(\mathcal K)$ such that
\begin{equation}
\label{Stines}
\Phi(X) = \fJ^*\pi (X)\fJ \text{ for every $X\in \mathcal C^*(I_\mathcal H, \mathcal T(P))$}.
\end{equation}
Note that $Q=\Phi(I_\mathcal H)=\fJ^*\fJ=\operatorname{SOT-}\lim_{n\rightarrow \infty} P^{*n}P^n.$

We shall need to go deeper into the properties of the Stinespring triple $(\cK,\pi,\fJ)$. The first property we get is

\vspace*{3mm}

\noindent (${\bf{ P_1}}$) {\em $U:=\pi(QP)$ is a unitary operator. Moreover, $\fJ P=U\fJ $ and $\mathcal K$ is the smallest reducing subspace for $U$ containing $\fJ \mathcal H$.}

\vspace*{3mm}

The proof is somewhat long. Since $\Phi$ has now been restricted to the $C^*$-algebra $C^*(I_\mathcal H, \mathcal T(P))$, its kernel is an ideal in $C^*(I_\mathcal H, \mathcal T(P))$ by Lemma~\ref{L:CE} (when $\Phi$ is allowed as a map on whole of $\mathcal B (\mathcal H)$, its kernel may not be an ideal). In view of the kernel of $\Phi$ being an ideal, it follows from the construction of the minimal Stinespring dilation that
$\text{Ker }\!\Phi= \text{Ker }\!\pi$. Thus
\begin{equation} \label{pix=piphix} \pi(X)=\pi(\Phi(X)) \text{ for any } X\in C^*(I,\mathcal T(P)). \end{equation}
This will be used many times.  Since $\pi$ is a representation, a straightforward computation  gives us
$$ U^*\pi(X)U = \pi(X) \text{ for any } X \in C^*(I,\mathcal T(P)).$$

Since $\pi$ is unital, we get that $U$ is an isometry. If $P'$ is a projection in the weak* closure of $\pi(C^*(I,\mathcal T(P)))$, then we also have $U^*P'U=P'$ and $U^* P'^{\perp}U=P'^{\perp}$. This shows that $UP'=P'U$ and therefore
\[\pi(X)U=U\pi(X)\] for all $X\in C^*(I,\mathcal T(P))$. In particular, it follows that $U$ is a unitary and
\[  \pi(C^*(I_\mathcal H,\mathcal T(P)))\subseteq\{U\}'. \]
We can harvest a quick crucial equality here, viz.,
\begin{equation} \label{thetaXVisVX} \pi(QX)\fJ =\fJ X \end{equation}
if $X \in \mathcal B(\mathcal H)$ commutes with $P$.

The proof of \eqref{thetaXVisVX} follows from two computations. For every $h, h^\prime \in \mathcal H$, we have
\begin{align*}
\langle \pi(QX)\fJ h,\fJ h^\prime \rangle&=\langle \fJ ^*\pi(QX)\fJ h,k\rangle\\
&=\langle \Phi(QX)h,h^\prime \rangle  \quad [\text{using } (\ref{Stines})]\\
&=\langle QXh,h^\prime \rangle \quad [\text{because } \mathcal T (P) \text{ is fixed by } \Phi] =\langle \fJ Xh,\fJ h^\prime \rangle \end{align*}
showing that $P_{\overline{\rm{Ran}} \fJ }\pi(QX)\fJ  = \fJ X$. On the other hand,
\begin{align*}
\|\pi(QX)\fJ h\|^2 &=\langle \fJ ^*\pi(X^*Q^2X)\fJ h,h \rangle \\
&=\langle \Phi(X^* Q^2X)h,h \rangle \\
&=\langle X^* \Phi(Q^2)Xh,h \rangle \quad [\text{by Lemma \ref{L:PJFA}}]\\
&=\langle X^*QXh,h\rangle \quad[\text{by Lemma \ref{L:CE}}] =\| \fJ Xh \|^2.
\end{align*}
Consequently, $\pi(QX)\fJ =\fJ X$ for every $X\in \{P\}'$. This, in particular, proves that $U\fJ=\fJ P$. To complete the proof of ${\bf{ P_1}}$, it is required to establish that $\mathcal K$ is the smallest reducing subspace for $U$ containing $\fJ \mathcal H$. To that end, we consider a map $\Gamma$ from Ran $\pi$ into $\mathcal T (P)$ given by $$\Gamma(\pi(X))=\fJ^*\pi(X)\fJ=\Phi(X) \text{ for all } X\in C^*(I, \mathcal T(P)).$$
It is injective because $\text{Ker }\!\Phi= \text{Ker }\!\pi$.

Since $\Gamma\circ \pi=\Phi$, we have $\Gamma\circ \pi$ to be idempotent and this coupled with the injectivity of $\Gamma$ gives us $\pi\circ\Gamma=I$ on $\pi\{C^*(I,\mathcal T(P))\}$. This immediately implies that $\Gamma$ is a complete isometry.

Let $\mathcal K_0\subseteq \mathcal K$ be the smallest reducing subspace for $U$ containing $\fJ \mathcal H$. Let $P_{\mathcal K_0}$ be the projection in $\mathcal B(\mathcal K)$ onto the space $\mathcal K_0$. Consider the vector space \[P_{\mathcal K_0}\{U\}'P_{\mathcal K_0}:=\{P_{\mathcal K_0}XP_{\mathcal K_0}: X\in\{U\}'\}= \{P_{\mathcal K_0}X|_{\mathcal K_0}\oplus 0_{\mathcal K_0^{\perp}}: X\in\{U\}'\}.  \]
and the map $\Gamma': P_{\mathcal K_0}\{U\}'P_{\mathcal K_0}\to \mathcal T(P)\subseteq \mathcal B(\mathcal H)$ defined by $X\mapsto \fJ^* X\fJ$. This is injective.

Indeed, it is easy to check that $\fJ^* X\fJ\in \mathcal T(P)$ for $X\in \{U\}'$. Now if $\fJ^*X\fJ=0$ for some $X\in \{U\}'$ then using the identity $\fJ P=U\fJ $, we get that
\[  \langle X f(U, U^*)\fJ h, g(U,U^*)\fJ k\rangle =0 \]
for any two variable polynomials $f$ and $g$ and $h,k\in\mathcal H$. This shows that $P_{\mathcal K_0}XP_{\mathcal K_0}=0$ and therefore, $\Gamma'$ is injective. For any $Y\in P_{\mathcal K_0}\{U\}'P_{\mathcal K_0}$,
\begin{align*}
 \Gamma'(P_{\mathcal K_0}\pi(\fJ^*Y\fJ)P_{\mathcal K_0}-Y)=\fJ^*\pi(\fJ^*Y\fJ)\fJ-\fJ^*Y\fJ=  \Phi(\fJ^*Y\fJ)-\fJ^*Y\fJ=0.
\end{align*}
Thus, by the injectivity of $\Gamma'$, we have
\[ P_{\mathcal K_0}\pi(C^*(I,\mathcal T(P)))P_{\mathcal K_0}=P_{\mathcal K_0}\{U\}'P_{\mathcal K_0}\]
In other words, we have a surjective complete contraction
\[ \tilde{C}_{\mathcal K_0}: \pi(C^*(I,\mathcal T(P))) \to P_{\mathcal K_0}\{U\}'P_{\mathcal K_0}= \{P_{\mathcal K_0}X|_{\mathcal K_0}\oplus 0_{\mathcal K_0^{\perp}}: X\in\{U\}'\},\]
defined by $X\mapsto P_{\mathcal K_0}X P_{\mathcal K_0}$. Since $\Gamma=\Gamma'\circ \tilde{C}_{\mathcal K_0}$ and $\Gamma$ is a complete isometry,
$\tilde{C}_{\mathcal K_0}$ is a complete isometry. Then the induced compression map
\[C_{\mathcal K_0}: \pi(C^*(I,\mathcal T(P))) \to \{P_{\mathcal K_0} U|_{\mathcal K_0}\}'\subseteq\mathcal B(\mathcal K_0),\quad  X\mapsto P_{\mathcal K_0}X|_{\mathcal K_0} \]
is a unital complete isometry and therefore a $C^*$-isomorphism by a result of Kadison (\cite{Kadison}). Hence by the minimality of the Stinespring representation $\pi$ we have
$\mathcal K=\mathcal K_0$ and therefore $\pi(C^*(I,\mathcal T(P))) =\{U\}'$. This not only completes the proof of ${\bf{ P_1}}$, but also proves

\vspace*{3mm}

\noindent (${\bf{ P_2}}$) {\em The map $\Gamma:\{U\}' \to \mathcal T(P)$ defined by $\Gamma (Y)=\fJ ^*Y\fJ $, for all $Y\in \{U\}'$, is surjective and a complete isometry.}

\noindent (${\bf{ P_3}}$) {\em The Stinesrping triple $(\cK,\pi,\fJ )$ satisfies $\pi\circ\Gamma=I$. In particular,
$$\pi(C^*(I_\mathcal H,\mathcal T(P)))=\{U\}'.$$}
\vspace*{3mm}
The final property that we shall need is

\vspace*{3mm}

\noindent (${\bf{ P_4}}$) The linear map $\Theta : \{P\}' \to  \{U\}'$
defined by $\Theta(X) = \pi(QX)$ is completely contractive, unital and multiplicative.

\vspace*{3mm}

To prove ${\bf{ P_4}}$, first note that $\Theta$ is completely contractive and unital as $\pi(Q)=I$. We have also proved that
$\Theta (X)\fJ=\fJ X$ for all $X\in\{P\}'$. Since, for $X,Y\in \{P\}'$,
\[ \Gamma(\Theta(XY)-\Theta(X)\Theta(Y))=\fJ^*\fJ XY-\fJ^*\Theta(X)\Theta(Y)\fJ=0,\]
then by injectivity of $\Gamma$ we have $\Theta$ is multiplicative and this completes the proof of ${\bf{ P_4}}$.

Since we have now developed the properties of the Stinespring dilation of $\Phi$ in detail, we are ready to complete the proof of the theorem. Define
$$U_i := \pi (QT_i) \mbox{ for } 1\leq i \leq d.$$
We observe that
$$U_1U_2\cdots U_d=\pi(QP)=U.$$ Indeed, using the property $({\bf P_4})$ above, we get
\begin{align*}
U=\pi(QP)=\Theta(P)&=\Theta(T_1)\Theta(T_2)\cdots\Theta(T_d)\\
&=\pi(QT_1)\pi(QT_2)\cdots\pi(QT_d)=U_1U_2\cdots U_d.
\end{align*}
Therefore each $U_j$ is a unitary operator.

That the triple $(\fJ,\cK,\underline{U}=(U_1,U_2,\dots,U_d))$ is actually a canonical pseudo-extension of $\underline{T}$ follows from \eqref{thetaXVisVX} when applied to $X=T_j$ for each $j=1,2,\dots,d$. Minimality of the pseudo-extension $\underline{U}$ follows from $({\bf{P_1}})$, which says that $\cK$ is actually equal to
\begin{align*}
\overline{\operatorname{span}}\{U^m\fJ h:h\in\cH \text{ and }m\in \mathbb Z\}.
\end{align*}
Let $\Gamma$ be as in $(\bf P_2)$ above. Note that
 \begin{align*}
 \{U_1,U_2,\dots,U_d\}'\subset\{U\}'.
 \end{align*} Consider the restriction of $\Gamma$ to $\{U_1,U_2,\dots,U_d\}'$ and continue to denote it by $\Gamma$. Since complete isometry is a hereditary property, to prove part (1),
 all we have to show is that $\Gamma(Y)$ lands in $\cT(\underline{T})$, whenever $Y$ is in $\{U_1,U_2,\dots,U_d\}'$ and $\Gamma$ is surjective.
 To that end, let $Y \in \{U_1,\dots,U_{n}\}'$.  Then for each $j=1,2,\dots,d$, we see that
\begin{eqnarray*}
T_j^*\Gamma(Y)T_j=T_j^*\fJ ^*Y\fJ T_j&=&\fJ ^*U_j^*YU_j\fJ =\fJ ^*Y\fJ =\Gamma(Y).
\end{eqnarray*}
Thus $\Gamma$ maps $\{U_1,\dots,U_{d}\}'$ into $\mathcal T(\underline{T})$. For proving surjectivity of $\Gamma$, let $X\in \mathcal T (\underline{T})$. This, in particular, implies that $X$ is in $\mathcal T(P)$. Applying $({\bf{ P_2}})$ again we have an $Y$ in $\{U\}'$ such that $\Gamma(Y)=\fJ ^*Y\fJ =X$.
It remains to show that this $Y$ commutes with each $U_j$. Since $X\in \mathcal T (\underline{T})$, we have
\begin{align*}
T_j^* XT_j=X \text{ for each } j=1,2,\dots, d
\end{align*}which is the same as $T_j^*\fJ ^*Y\fJ T_j = \fJ ^*Y\fJ$. Applying the intertwining property of $\fJ$, we get for each $j$
\begin{align*}
\fJ ^*U_j^*YU_j\fJ  = \fJ ^*Y\fJ
\end{align*} which is the same as $\Gamma(U_j^*YU_j-Y) =0$ for each $j$. Since $\Gamma$ is an isometry, the commutativity of $Y$ with each $U_j$ is established. This completes the proof of part (1).

Part (2) of the Theorem follows from the content of $({\bf{P_3}})$ if we restrict $\pi$ to ${C^*(I,\mathcal T(\underline{T}))}$ and continue to call it $\pi$.

For the last part of theorem, let us take $\Theta$ as in $({\bf{P_4}})$, i,e.,
$$\Theta(X)=\pi(QX)$$ for every $X$ in $\{P\}'$. Restrict $\Theta$ to ${\{T_1,\dots,T_{d}\}'}$ and continue to call it $\Theta$. The aim is to show that $\Theta(X)\in \{U_1, \ldots ,U_{d}\}'$ if $X\in \{T_1,\dots,T_{d}\}'$.
For this we first observe that if $X$ commutes with each $T_j$, then $QX$ is in $\cT(\underline{T})$.
Now the rest of the proof follows from part (2) of the theorem and \eqref{thetaXVisVX}.
\end{proof}

\vspace{0.1in} \noindent\textbf{Acknowledgement:}
The first named author's research is supported by the University Grants Commission Centre for Advanced Studies. The research works of the second and third named authors are supported by DST-INSPIRE Faculty Fellowships DST/INSPIRE/04/2015/001094 and DST/INSPIRE/04/2018/002458 respectively.

\end{document}